\documentclass[a4paper,final]{siamltex}
\usepackage{amsmath,amssymb}
\usepackage{graphicx,epsfig}
\usepackage{psfrag}
\usepackage{color}
\usepackage{amsfonts}
\usepackage{latexsym}
\usepackage{pstricks}
\usepackage{bbm}

\renewcommand{\theequation}{\thesection.\arabic{equation}}
\makeatletter\@addtoreset{equation}{section}\makeatother
\newtheorem{remark}[theorem]{Remark}

\newtheorem{assumption}[theorem]{Assumption}

\newcounter{constantsnumber}

\newcommand{\CT}{\mathcal{T}}
\newcommand{\CV}{\mathcal{V}}
\newcommand{\CJ}{\mathcal{J}}
\newcommand{\CC}{\mathcal{C}}

\newcommand{\bn}{\mbox{\boldmath{$n$}}}

\newcommand{\CW}{\mathcal{W}}

\newcommand{\BBR}{\mbox{$\mathbb{R}$}}
\newcommand{\BBN}{\mbox{$\mathbb{N}$}}
\newfont{\twelvemsb}{msbm10 at 11.6pt}

\newcommand{\tM}{\mathtt{M}}
\newcommand{\tA}{\mathtt{A}}
\newcommand{\tD}{\mathtt{D}}
\newcommand{\tB}{\mathtt{B}}

\usepackage{colordvi}

\title{A mixed finite element method for a biharmonic problem
with weakly imposed Dirichlet boundary condition}

\author{Bishnu P.~Lamichhane\thanks{
School of Information and Physical Sciences, 
 University of Newcastle, 
NSW 2308, Callaghan, {\tt Bishnu.Lamichhane@newcastle.edu.au}}}

\begin{document}
\maketitle

\begin{abstract}
We consider a mixed finite element method for a biharmonic equation with 
clamped boundary conditions  based on 
 biorthogonal systems with 
weakly imposed Dirichlet boundary condition. 
We show that the weak imposition of the boundary condition 
arising from a natural minimisation formulation allows 
to get an optimal a priori error estimate for the finite element scheme improving 
the existing error estimate for such a formulation without weakly 
imposed Dirichlet boundary condition.
We also briefly outline the algebraic formulation arising from the finite element 
method.
\end{abstract}



\pagestyle{myheadings}
\thispagestyle{plain}

\section{Introduction}
Thin plates and beams, strain gradient elasticity,   phase separation of a binary mixture and 
fluid flow problems are often modelled by 
fourth order elliptic and parabolic problems \cite{Cia78,EGH02,GR86, WKG06}.
This difficulty of constructing $H^2$ - conforming finite element 
spaces is avoided 
either by using  a discontinuous Galerkin method as in \cite{EGH02,BS05,WKG06}
or by using a mixed formulation as in 
\cite{CR74,CG75,Fal78,Cia78,FO80,BOP80,Mon87,DP01}.

In this paper, we start with a  
mixed finite method due to Ciarlet and Raviart \cite{CR74,CG75,Cia78} using 
different spaces for the stream function and vorticity for a 
fourth order problem with clamped boundary conditions. 
The great advantage of this formulation is that it 
 allows the use of the standard $H^1$-conforming finite element method. 
Working with this formulation for clamped boundary conditions
the a priori error estimate is sub-optimal
\cite{CR74,Cia78,Sch78, GR86, DP01,Lam11a, Zul15}, 
where the finite element method of order $k$ converges with $h^{k-\frac{1}{2}}$ in 
the energy norm. The strong imposition of the Dirichlet boundary condition is 
the main reason for the sub-optimal convergence rate. 
In order to get an optimal estimate, we impose the Dirichlet boundary condition 
weakly using a Nitsche type approach. This leads to an optimal order of convergence
improving the existing a priori error estimate for the biharmonic 
problem with clamped boundary conditions. 
 As in \cite{Lam11a} 
we work with discrete spaces having local basis functions satisfying 
the condition of biorthogonality
for the discretisation 
of the stream function and vorticity. This yields a very 
efficient finite element method to approximate the solution 
of a fourth order problem. While 
the standard symmetric Nitsche apporach requires a penalty parameter \cite{Ste95}, 
 our approach does not require a  penalty parameter.

The structure of the rest of the paper is organised as follows. 
In the rest of this section, we briefly recall a mixed formulation for 
a  biharmonic equation with clamped boundary conditions and 
extend the formulation to include non-homogeneous clamped boundary conditions. 
Section \ref{sec:ana} is devoted for the numerical analysis of the approach. 
We give an algebraic formulation of the finite element scheme in Section \ref{sec:alg}. 
Finally, we draw a conclusion in the last section.

\subsection{Mixed formulation} 
We now derive a mixed formulation of a fourth order problem.
We first briefly recall a mixed formulation of the 
biharmonic problem with homogeneous clamped boundary conditions.
\paragraph{Homogeneous clamped boundary conditions}
 Let  $\Omega \subset \BBR^2$
be a bounded convex domain 
with polygonal boundary $\Gamma = \partial\Omega$ and 
outward pointing normal $\bn$ on  $\Gamma$.
We consider the  biharmonic equation 
\begin{equation}\label{biharm}
 \Delta^2 u  =f \quad\text{in}\quad \Omega
\end{equation}
 with clamped boundary conditions
\begin{equation}\label{clbc}
u =\frac{\partial u}{\partial \bn}=0\quad\text{on}\quad \Gamma.
\end{equation}

Following the same approach as in \cite{CR74,Cia78,Lam11a} 
we recast the biharmonic problem as a minimisation problem with a constraint 
and then reformulate the problem as a three-field formulation. 
The main idea here is to include the weak form of the Dirichlet boundary condition. 
We note that the main difficulty to get optimal error estimates using simplicial Lagrange 
finite element methods for the biharmonic problem is the imposition of 
the Dirichlet boundary condition on the boundary in the strong sense, which induces a loss of 
accuracy in the error estimates. To rectify this we propose to 
impose the Dirichlet boundary condition weakly using a minimisation 
formulation or equivalently Nitsche approach. 
In contrast to other Nitsche approaches,
we do not require a penalty parameter in our formulation.

We use usual notations for Sobolev spaces as \cite{LM72,Ada75,Gri85,BS94}.
We consider the following variational form of the biharmonic problem 
\begin{equation}\label{vbiharm}
J(u)=\inf_{v \in H^2_0(\Omega)}J(v),
\end{equation}
with 
\begin{equation}\label{func}
J(v)=\frac{1}{2} \int_{\Omega}|\Delta v|^2\,dx-\int_{\Omega}f\,v\,dx.
\end{equation}


Let $H^*(\Omega)$ be the dual space of $H^1(\Omega)$.
We now introduce a new unknown 
$\phi = \Delta u$ and write a weak form of this equation as 
\[ \int_{\Omega}\phi\mu \,dx - \langle u,\Delta \mu \rangle=0, 
\quad \mu \in Q,\]
where $\langle u,\Delta \mu \rangle$ is the duality pairing between 
the spaces $H^1(\Omega)$ and its dual $H^*(\Omega)$, and \[ 
Q =\{v\in H^1(\Omega):\, \int_{\Omega} v\,dx =0\}.\]
This is a right choice for the Lagrange multiplier space as 
\[ \int_{\Omega} \phi\,dx =0.\]
Let $V = H^1(\Omega) \times L^2(\Omega)$.
The variational problem \eqref{vbiharm}
can be recast as the minimization problem \cite{Cia78}
\begin{equation}\label{cmbiharm}
\CJ(u,\phi)=\inf_{(v,\psi) \in \CV}\CJ(v,\psi),
\end{equation}
where 
\begin{eqnarray*}
\CJ(v,\psi)&=&\frac{1}{2} \int_{\Omega}|\psi|^2\,dx+ \frac{1}{2} \|v\|^2_{\frac{1}{2},\Gamma}-\int_{\Omega}f\,v\,dx,\\
\CV&=&\{(v,\psi)\in V:\;\int_{\Omega}\psi\,q\,dx -\langle v,\Delta q\rangle =0,\;
q \in Q\}.
\end{eqnarray*}
In the following, the $H^{\frac{1}{2}}(\Gamma)$ inner product is 
denoted by $\langle \cdot,\cdot\rangle_{\frac{1}{2},\Gamma}$ and $H^{\frac{1}{2}}$-norm 
by $\|\cdot \|^2_{\frac{1}{2},\Gamma}$. The dual space of 
$H^{\frac{1}{2}}(\Gamma)$ is denoted by  $H^{-\frac{1}{2}}(\Gamma)$.

\paragraph{Non-homogeneous boundary conditions}

In the following, we consider the biharmonic problem \eqref{biharm} with 
non-homogeneous clamped boundary conditions
 with $g_D\in H^{\frac{1}{2}}(\Gamma),\; g_N \in H^{-\frac{1}{2}}(\Gamma)$.
These boundary conditions are as follows:
\begin{equation}\label{nonhom}
u =g_D\quad\text{and}\quad \frac{\partial u}{\partial \bn}=g_N\quad\text{on}\quad \Gamma.
\end{equation}
Then, we have the minimisation problem \eqref{cmbiharm} with 
\begin{eqnarray*}
\CJ(v,\psi)&=&\frac{1}{2} \int_{\Omega}|\psi|^2\,dx+ \frac{1}{2} \|v-g_D\|^2_{\frac{1}{2},\Gamma}-\int_{\Omega}f\,v\,dx,\\
\CW&=&\{(v,\psi)\in V:\;\int_{\Omega}\psi\,q\,dx -\langle v,\Delta q\rangle =
\langle g_N,q \rangle_{\Gamma}- \langle   \frac{\partial q}{\partial\bn} ,g_D \rangle_{\Gamma},\;
q \in Q\},
\end{eqnarray*}
where  $\langle \cdot,\cdot \rangle_{\Gamma}$ is the duality pairing between the spaces
$H^{\frac{1}{2}}(\Gamma)$ and  its dual $H^{-\frac{1}{2}}(\Gamma)$ 
\begin{remark}
Here, the normal derivative of an $H^1$-function is a generalised 
normal derivative as defined in \cite{Mik11,mclean}. Lemma 4.3 of \cite{mclean} gives the following bound for the normal derivative of $q \in H^1(\Omega)$ (see also \cite{Mik11})
\[ \|\frac{\partial q}{\partial\bn} \|_{-\frac{1}{2},\Gamma} \leq C ( \|q\|_{1,\Omega} + 
\|\Delta q\|_{H^*(\Omega)}).\]
\end{remark}

The problem \eqref{cmbiharm} can be recast as 
a saddle point formulation \cite{Lam11a,CR74,Cia78,DP01}. The saddle point 
problem is:  Given $\ell \in H^{-1}(\Omega)$, find 
$((u,\phi),p) \in V\times Q$ such that 
\begin{equation}\label{saddle}
\begin{array}{llccc}
a((u,\phi),(v,\psi))+&b((v,\psi),p)&=&\ell(v),&
\quad (v,\psi) \in  V,\\
b((u,\phi),q)&&=&g(q),&\quad q\in Q,
\end{array}
\end{equation}
where 
\begin{eqnarray}\label{saddledef}
 &&a((u,\phi),(v,\psi)) =
 \int_{\Omega}\phi\psi\,dx+\langle u, v \rangle_{\frac{1}{2},\Gamma}, \\ \nonumber
&& \ell(v) = \int_{\Omega} fv\,dx +
\langle g_D,v\rangle_{\frac{1}{2},\Gamma},\quad
b((v,\psi),q)= \int_{\Omega}\psi\,q\,dx-\langle v,\Delta q\rangle,
\\ \nonumber  && \text{and}\quad 
g(q) =  \langle g_N,q \rangle_{\Gamma} - \langle   \frac{\partial q}{\partial\bn} ,g_D \rangle_{\Gamma} .
\end{eqnarray}

\paragraph{Consistency}
Let $u\in H^2(\Omega)$ be the solution of the 
biharmonic problem \eqref{biharm} with 
the non-homogeneous boundary conditions \eqref{nonhom}.
Let $\phi = \Delta u$ and $p = - \phi$. 
An integration by parts can be performed to show that 
they satisfy the saddle point equations \eqref{saddle}.

\begin{remark}
[\bf Existence and uniqueness of the solution]
There is a difficulty in proving the coercivity of the bilinear form $a(\cdot,\cdot)$ in 
the saddle point problem \eqref{saddle} as the standard trace theorem \cite{Gri85} 
does not work for the generalised normal derivative \cite{mclean,Mik11}. However, there is 
no problem for defining the standard normal derivative for a function $q_h$ in the standard 
finite element space, see 
the next section. Therefore, we do not analyse the existence and uniqueness of 
the saddle point problem \eqref{saddle}, but rather focus on 
 its discrete counterpart in the following section.
 \end{remark}

\section{Finite element discretizations}\label{sec:ana}
We consider a quasi-uniform and shape-regular triangulation $\CT_h$ of the 
polygonal domain $\Omega$ with the global mesh-size $h$, where $\CT_h$
consists of triangles or parallelograms. 
Let $\CC_h$ be the collection of boundary edges of the triangulation of $\Omega$. 
We use $h_K$ and $h_e$ to denote the sizes of the elements in 
$\CT_h$ and $\CC_h$, respectively. 
Let $S_h \subset H^1(\Omega)$ be a standard Lagrange finite element 
space of order $k\in \BBN$, and $M_h \subset L^2(\Omega)$ 
be another piecewise polynomial space. 
We also set $V_h = S_h\times M_h$.
We have  a well-known approximation result for every $u \in H^{k+1}(\Omega)$ \cite{Bra01}: there exists a function $u_h \in S_h$ such that
\[ h \|u-u_h\|_{1,\Omega} + \|u-u_h\|_{0,\Omega} \leq C h^{k+1} \|u\|_{k+1,\Omega}.\]
In the following, we use a generic constant $C$, which takes 
different values in different occurrences but is always independent of the 
mesh-size. 
We impose the following assumptions on $M_h$.
\begin{assumption} \label{ass0}
We assume that there is a constant $C>0$ independent of 
the mesh-size such that 
\begin{eqnarray}
\|q_h\|_{0,\Omega} \leq C \sup_{\phi_h \in S_h} 
\frac{\int_{\Omega} \phi_h q_h\,dx} {\|\phi_h\|_{0,\Omega}},
\quad q_h \in M_h,
\end{eqnarray}
\end{assumption}
\begin{assumption} \label{ass1}
The space $M_h$ has the approximation property:
\begin{equation}
\inf_{\lambda_h \in M_h}\|\phi-\lambda_h\|_{0,\Omega}\leq Ch^k |\phi|_{k,\Omega},\quad \phi \in H^k(\Omega).
\end{equation} 
\end{assumption}
We use
\[ Q_h =\{ v_h \in S_h:\, \int_{\Omega} v_h\,dx =0\}\]
 to approximate the Lagrange multiplier space $Q$.
Our analysis is based on the following mesh-dependent inner product and the norm induced 
by this inner product on the boundary of $\Omega$ for $s \in [-1,1]$ \cite{Ste95}:
\begin{equation}\label{ip}
  \langle v, \, w \rangle_{s,h} = \sum_{e \in \CC_h} \frac{1}{h_e^{2s}} \int_{e} v\,w\,d\sigma,\quad  v,w \in L^2(\Omega).
  \end{equation}
We will use the mesh-dependent norm for $v_h \in S_h$, 
\[ \|v_h\|_{1,h}^2 = \|v_h\|_{1,\Omega}^2 + \|v_h\|_{\frac{1}{2},h}^2,\]
 where $\|\cdot\|_{\frac{1}{2},h}$ is  the norm induced by the inner product \eqref{ip}.
 In fact,  \[ \|u_h\|_{\frac{1}{2},h}^2  = \sum_{e \in \CC_h} \frac{1}{h_e} \int_{e} u_h^2\,d\sigma.\]

With the definition of $\|\cdot\|_{s,h}$-norm we have 
the following Cauchy-Schwarz type inequality for the inner product $\langle \cdot, \cdot \rangle_{\frac{1}{2},h} $ [3.13 of \cite{Ste95}]:
\begin{eqnarray}\label{cst}
\langle v, \, w \rangle_{\frac{1}{2},h} \leq 
\|v\|_{\frac{1}{2},h} \|w\|_{-\frac{1}{2},h}, \quad v \in H^1(\Omega),\; w \in L^2(\Omega).
\end{eqnarray}

The  discrete biharmonic problem is given as a saddle point problem: 
given $f \in H^{-1}(\Omega)$, $g_D\in H^{\frac{1}{2}}(\Gamma),\; g_N \in H^{-\frac{1}{2}}(\Gamma)$, 
find 
$((u_h,\phi_h),p_h) \in V_h\times S_h$ such that 
\begin{equation}\label{dbiharm}
\begin{array}{llccc}
a_h((u_h,\phi_h),(v_h,\psi_h))+&b_h((v_h,\psi_h),p_h)&=&\ell_h(v_h),&
\quad (v_h,\psi_h) \in  V_h,\\
b_h((u_h,\phi_h),q_h)&&=&g_h(q_h),&\quad q_h\in Q_h,
\end{array}
\end{equation}
where 
\begin{eqnarray*}\label{saddlehdef}
 a_h((u_h,\phi_h),(v_h,\psi_h)) &=& 
 \int_{\Omega}\phi_h\psi_h\,dx+ \langle u_h, \, v_h \rangle_{\frac{1}{2},h}, \;
b_h((v_h,\psi_h),q_h)= \int_{\Omega}\psi_h\,q_h\,dx - \langle v_h, \Delta_h q_h\rangle \\ 
\ell_h(v_h) &=& \int_{\Omega} fv_h\,dx +  \langle g_D,\, v_h \rangle_{\frac{1}{2},h} \quad\text{and}\quad 
g_h(q_h) =  \langle g_N,q_h \rangle_{\Gamma} - \int_{\Gamma}  \frac{\partial q_h}{\partial\bn} g_D\,d\sigma,
\end{eqnarray*}
where for  $ q\in H^{\frac{3}{2} + \epsilon}(\Omega)$ with $\epsilon>0$,  
$\Delta_h:  H^{\frac{3}{2} + \epsilon}(\Omega) \to M_h$ is defined as 
\[\langle v_h, \Delta_h q \rangle  = -\int_{\Omega} \nabla v_h \cdot \nabla q \,dx  +
\int_{\Gamma}   \frac{\partial q}{\partial\bn}  v_h\,d\sigma,\quad v_h \in S_h.\]
We note that $\Delta_h q$ is well-defined due to Assumption \ref{ass0}. 


In order to analyse the finite element problem we introduce the mesh-dependent graph norm 
on $V_h$ defined as 
\begin{eqnarray}\label{meshnorm}
\|(v_h,\psi_h)\|_{a} = \sqrt{ \|\psi_h\|^2_{0,\Omega}  + \|v_h\|_{1,h}^2}
\end{eqnarray}
and the following mesh-dependent norm for the Lagrange multiplier $q_h \in Q_h$ defined as
\[ \|q_h\|^2_{Q_h} = \|q_h\|^2_{0,\Omega} + \|\Delta_h q_h\|_{-1,h}^2,\]
where \[  \|\Delta_h q_h\|_{-1,h} = 
\sup_{v_h \in S_h} \frac{\langle \Delta_h q_h, v_h\rangle}{\|v_h\|_{1,h}}.\]

We can see that the continuity of the bilinear form $a_h(\cdot,\cdot)$ and linear forms 
$\ell_h(\cdot)$ and $g_h(\cdot)$ 
follows from the Cauchy-Schwarz and trace inequalities \cite{FS95}.
The continuity of the bilinear form $b_h(\cdot,\cdot)$ follows from  
\[ \|w_h\|_{1,h} \|\Delta_h q_h\|_{-1,h} =  \|w_h\|_{1,h}
\sup_{v_h \in S_h} \frac{\langle \Delta_h q_h, v_h\rangle}{\|v_h\|_{1,h}} 
\geq \left|\langle \Delta_h q_h, w_h\rangle\right|, \; 
w_h\in S_h,\, q_h \in Q_h.\]
Thus 
\[ |b_h((w_h,\psi_h),q_h) | \leq \|\psi_h\|_{0,\Omega} \|q_h\|_{0,\Omega} + \|w_h\|_{1,h} \|\Delta_h q_h\|_{-1,h} .\]
We now show the inf-sup condition for the bilinear form $b_h(\cdot,\cdot)$.  We need to show the existence 
of a mesh-independent constant $C$ such that 
\begin{equation}\label{infsup}
\sup_{(v_h,\psi_h) \in V_h} \frac{b_h((v_h,\psi_h),q_h)}{\|(v_h,\psi_h)\|_a} \geq C \|q_h\|_{Q_h}.\end{equation}
First we set $v_h=0$ on the  left hand side of the above inequality and use \eqref{ass0} to obtain 
\[ \sup_{(v_h,\psi_h) \in V_h} \frac{b_h((v_h,\psi_h),q_h)}{\|(v_h,\psi_h)\|_a} \geq 
\sup_{\psi_h \in M_h} \frac{\int_{\Omega} q_h \psi_h}{\|\psi_h\|_{0,\Omega}} \geq C \|q_h\|_{0,\Omega}.\]
In the second step, we set $\psi_h = 0$  on the  left hand side of the inequality 
\eqref{infsup} and use the definition of the norm $\|\cdot\|_{-1,h}$ to obtain 
\[\sup_{(v_h,\psi_h) \in V_h} \frac{b_h((v_h,\psi_h),q_h)}{\|(v_h,\psi_h)\|_a} \geq 
\sup_{v_h \in S_h} \frac{\langle v_h, \Delta_h q_h \rangle }{\|v_h\|_{1,h}}  = \|\Delta_h q_h\|_{-1,h}.\]
Now we turn our attention to prove the coercivity of the bilinear form $a_h(\cdot,\cdot)$ 
on the kernel space  $\CV_h$ defined as 
\begin{equation}\label{kernel}
\CV_h= \{(v_h,\psi_h)\in V_h:\;\int_{\Omega}\psi_h\,q_h\,dx-\langle \Delta_h q_h, v_h\rangle=0,\;q_h \in Q_h\}.
\end{equation}
First, we note that 
\[ a_h((v_h,\psi_h),(v_h,\psi_h)) = \|\psi_h\|^2_{0,\Omega} + \|v_h\|^2_{\frac{1}{2},h}.\]
If $(v_h,\psi_h) \in \CV_h$, we have 
\begin{equation}\label{keq}
 \int_{\Omega} \left(\psi_h\,q_h + \nabla v_h \cdot \nabla q_h\right)\,dx  = 
\int_{\Gamma} \frac{\partial q_h}{\partial \bn} v_h\,d\sigma,\quad q_h \in Q_h.
\end{equation}
Let 
\[ q_h = v_h - \frac{1}{|\Omega|} \int_{\Omega} v_h\,dx \in Q_h.\]
Then we have 
\[ \frac{\partial q_h}{\partial \bn}  = \frac{\partial v_h}{\partial \bn} \quad\text{and}\quad 
\nabla q_h = \nabla v_h.\]
Hence for $(v_h,\psi_h) \in \CV_h$, using this  $q_h$ in  
\eqref{keq}, we 
obtain 
\begin{equation}\label{eqn10}
\|\nabla v_h\|_{0,\Omega}^2 = \int_{\Gamma} \frac{\partial v_h}{\partial \bn} v_h\,d\sigma- 
\int_{\Omega} \psi_h\,\left(v_h - \frac{1}{|\Omega|} \int_{\Omega} v_h\,dx\right) \,dx.
\end{equation}
We now apply the Cauchy-Schwarz type inequality for the boundary integral of 
the first term on the right of the above equation 
\[ \left|\int_{\Gamma} \frac{\partial v_h}{\partial \bn} v_h\,d\sigma \right| \leq 
\left\| \frac{\partial v_h}{\partial \bn} \right\|_{-\frac{1}{2},h} \left\| v_h\right\|_{\frac{1}{2},h},\]
so that  \eqref{eqn10} yields
\begin{equation}\label{eqn11}
  \|\nabla v_h\|_{0,\Omega}^2  \leq \left\| \frac{\partial v_h}{\partial \bn} \right\|_{-\frac{1}{2},h} \left\| v_h\right\|_{\frac{1}{2},h} + \|\psi_h\|_{0,\Omega} \left\|v_h - \frac{1}{|\Omega|} \int_{\Omega} v_h\,dx\right\|_{0,\Omega}.
  \end{equation}
In terms of the following trace inequality [(4) of \cite{FS95}]
\[ \left\| \frac{\partial v_h}{\partial \bn} \right\|_{-\frac{1}{2},h}  \leq C\|\nabla v_h\|_{0,\Omega},\]
and Poincar\'e-Friedrichs inequality 
\[ \left\|v_h - \frac{1}{|\Omega|} \int_{\Omega} v_h\,dx\right\|_{0,\Omega} \leq C \|\nabla v_h\|_{0,\Omega},\]
we get  from \eqref{eqn11}
\[  \|\nabla v_h\|_{0,\Omega}^2  \leq  C \left( \|\nabla v_h\|_{0,\Omega}\left\| v_h\right\|_{\frac{1}{2},h} +
\|\psi_h\|_{0,\Omega} \|\nabla v_h\|_{0,\Omega}\right).\]
Hence we have 
\[ \|\nabla v_h\|_{0,\Omega} \leq C( \|\psi_h\|_{0,\Omega} + \left\| v_h\right\|_{\frac{1}{2},h}).\]
Moreover, we have a mesh-independent constant $C$ such that  \cite{Bre03}
\[ \|v_h\|_{0,\Omega} \leq C( \|\nabla v_h\|_{0,\Omega} + \left\| v_h\right\|_{\frac{1}{2},h}).\]

Thus we have the following lemma for the 
coercivity of the bilinear form $a_h(\cdot,\cdot)$ on $\CV_h$.
\begin{lemma}\label{lemma1}
 There exists $\alpha_0>0$ independent of the mesh-size $h$  such that 
\[ a_h((v_h,\psi_h),(v_h,\psi_h))\geq \alpha_0 
(\|v_h\|^2_{1,h}+\|\psi_h\|^2_{0,\Omega}),\; (v_h,\psi_h) \in \CV_h.
\] 
\end{lemma}
Hence we have obtained the well-posedness of the saddle point problem \eqref{dbiharm}. 
\begin{lemma}\label{lemma1}
The saddle point problem \eqref{dbiharm} has a unique solution 
$((u_h,\phi_h),p_h) \in V_h\times S_h$.
\end{lemma}
We use the following lemma to prove the 
a priori error estimate for the discrete solution \cite{Lam11a}.

\begin{lemma}\label{lemma2}
Let $u$ be the solution of the biharmonic equation \eqref{biharm} with 
non-homogeneous boundary condition \eqref{nonhom}, and 
$\phi = \Delta u$ as well as $p = -\phi$. Let $p \in H^{k+1}(\Omega)$.  Let 
$((u_h,\phi_h),p_h) \in V_h\times Q_h$ be the solution of the discrete 
problem \eqref{dbiharm}. Then there exists a constant $C>0$ independent of 
the mesh-size $h$ so that 
\begin{equation}\label{est1}
\|(u-u_h,\phi-\phi_h)\|_{a} \leq 
C\left(\inf_{(w_h,\xi_h)\in \CW_{h}} \|(u-w_h,\phi-\xi_h)\|_{a}+h^k \|p\|_{k+1,\Omega}\right),
\end{equation}
where 
\[ \CW_h=\{(w_h,\xi_h)\in V_h|\,  \int_{\Omega} \xi_hq_h\,dx - \langle \Delta_h q_h,w_h\rangle = 
\langle g_N,q_h \rangle_{\Gamma} - \langle   \frac{\partial q_h}{\partial\bn} ,g_D \rangle_{\Gamma} ,\;
q_h \in Q_h \}.\]
\end{lemma}
\begin{proof}
Let $(w_h,\xi_h) \in \CW_h$. 
Then $(w_h,\xi_h) $ satisfies
\[ \int_{\Omega} \xi_hq_h\,dx - \langle \Delta_h q_h,w_h\rangle = 
\langle g_N,q_h \rangle_{\Gamma} - \langle   \frac{\partial q_h}{\partial\bn} ,g_D \rangle_{\Gamma} ,\;
q_h \in Q_h.\]
Thus \eqref{dbiharm} implies 
$(u_h-w_h,\phi_h-\xi_h) \in \CV_h$, and hence coercivity of $a_h(\cdot,\cdot)$ on $\CV_h$ yields 
\[ \alpha_0 \|(u_h-w_h,\phi_h-\xi_h)\|_{a}
 \leq \sup_{(v_h,\psi_h)\in \CV_h} 
\frac{a_h((u_h-w_h,\phi_h-\xi_h),(v_h,\psi_h))}{\|(v_h, \psi_h)\|_{a}}.
\] 
Since from \eqref{dbiharm} and \eqref{saddle} 
$a_h((u-u_h,\phi-\phi_h),(v_h,\psi_h))+ b_h((v_h,\psi_h),p)=0$ for all 
$(v_h,\psi_h) \in \CV_h$, we have 
\begin{eqnarray*} 
a_h((u_h-w_h,\phi_h-\xi_h),(v_h,\psi_h))&=& 
a_h((u-w_h,\phi-\xi_h),(v_h,\psi_h))+a_h((u_h-u,\phi_h-\phi),(v_h,\psi_h))\\
&=&a_h((u-w_h,\phi-\xi_h),(v_h,\psi_h))+ b_h((v_h,\psi_h),p).
\end{eqnarray*}
Let $\tilde p_h \in Q_h$ be a  finite element interpolant for $p$. 
Using the fact that 
 \[ b_h((v_h,\psi_h),p) = 
\int_{\Omega} \psi_h p \,dx + \int_{\Omega} \nabla p\cdot \nabla v_h\,dx - 
\langle  \frac{\partial p}{\partial\bn} ,v_h \rangle_{\Gamma},\;\;\text{and}\;\;
(v_h,\psi_h) \in \CV_h,\] we get 
\[ b_h((v_h,\psi_h),p)= b_h((v_h,\psi_h),p-\tilde p_h)= 
\int_{\Omega} \psi_h (p-\tilde p_h) \,dx + \int_{\Omega} \nabla (p-\tilde p_h)\cdot \nabla v_h\,dx - 
\langle  \frac{\partial (p-\tilde p_h)}{\partial\bn} ,v_h \rangle_{\Gamma}.
\]
\noindent
We note that  the interpolant $\tilde p_h$ satisfies \cite[Lemma 2.3]{ThomeeBook}
\[  \left|\langle  \frac{\partial (p-\tilde p_h)}{\partial\bn} ,v_h \rangle_{\Gamma} \right|
\leq h^k \|p\|_{k+1,\Omega}\|v_h\|_{\frac{1}{2},h}.\]
And hence
\[ |b_h((v_h,\psi_h),p)|\leq C h^k \|p\|_{k+1,\Omega}\,\|(v_h, \psi_h)\|_{a}.\]
Thus 
\begin{eqnarray*}
\alpha_0 \|(u_h-w_h,\phi_h-\xi_h)\|_{a}
&\leq& \sup_{(v_h,\psi_h)\in \CV_h} 
\frac{a_h((u-w_h,\phi-\xi_h),(v_h,\psi_h))}{\|(v_h, \psi_h)\|_{a}}
+C h^k \|p\|_{k+1,\Omega} \\
&\leq& \|(u-w_h,\phi-\xi_h)\|_{a}+C h^k \|p\|_{k+1,\Omega},
\end{eqnarray*}
where we have used the fact that the continuity constant of the bilinear form $a(\cdot,\cdot)$ is 1.
Finally, a triangle inequality yields the estimate \eqref{est1}:
\begin{eqnarray*}
\|(u-u_h,\phi-\phi_h)\|_{a} &\leq&  
\|(u-w_h,\phi-\xi_h)\|_{a} + \|(w_h-u_h,\xi_h-\phi_h)\|_{a}\\
&\leq& \left(1+\frac{1}{\alpha_0}\right)\|(u-w_h,\phi-\xi_h)\|_{a}+
\frac{C}{\alpha_0} h^k \|p\|_{k+1,\Omega}.
\end{eqnarray*} 
\end{proof}

\begin{theorem}\label{th1}
Let $u$ be the solution of the biharmonic equation \eqref{biharm} with 
non-homogeneous boundary condition \eqref{nonhom}, and 
$\phi = \Delta u$ as well as $p = -\phi$. Let 
$((u_h,\phi_h),p_h) \in V_h\times Q_h$ be the solution of the discrete saddle point 
problem \eqref{dbiharm}. Let $u \in H^{k+1}(\Omega)\cap H^1_0(\Omega)$,
$\phi \in H^{k}(\Omega), \, p \in H^{k+1}(\Omega)$, and  
Assumptions \eqref{ass0} and \eqref{ass1} are satisfied. 
Then there exists a constant $C>0$ independent of 
the mesh-size $h$ so that 
\begin{equation}\label{est2}
\|(u-u_h,\phi-\phi_h)\|_{a} \leq 
C h^k \left(\|u\|_{k+1,\Omega}+|\phi|_{k,\Omega}+\|p\|_{k+1,\Omega}\right).
\end{equation}
\end{theorem}
\begin{proof}
Let $\Pi_h:L^2(\Omega) \to M_h$ and $\Pi^*_h:L^2(\Omega) \to S_h$ 
be two projections defined by 
\[ \int_{\Omega}  \Pi_h v \, q_h\,dx = \int_{\Omega} v \, q_h\,dx,\; q_h \in S_h,\quad\text{and}\]
\[ \int_{\Omega}  \Pi^*_h v \, \eta_h\,dx = \int_{\Omega} v \, \eta_h\,dx,\; \eta_h \in M_h.\]
These projectors are well-defined by Assumption \ref{ass0}.
Moreover, using Assumptions \ref{ass0} and \ref{ass1} we have  \cite{LW04a}
\begin{equation}
\label{eqn21}
\|\Pi_h v\|_{0,\Omega} \leq C \|v\|_{0,\Omega},\;\text{and}\;
\|\Pi_h w - w\|_{0,\Omega} \leq Ch^k \|w\|_{k,\Omega}
\;\text{for}\;  v \in L^2(\Omega),\;\text{and}\;  w \in H^k(\Omega). 
\end{equation}
\noindent
 Similarly, for $v \in L^2(\Omega)$ and $w \in H^1(\Omega)$, we have 
 \cite{LW04a} 
 \begin{equation}\label{eqn22}
\|\Pi^*_h v\|_{0,\Omega} \leq C \|v\|_{0,\Omega},\quad\text{and}\quad 
\|\Pi^*_h w \|_{1,\Omega} \leq C \|w\|_{1,\Omega}.
\end{equation}
We also have  for $r = \{0,1\}$ and $w \in H^{k+1}(\Omega)$
\begin{equation}\label{eqn21}
\|\Pi^*_h w - w\|_{r,\Omega} \leq Ch^{k+1-r} \|w\|_{k+1,\Omega}.
\end{equation}
\noindent 
Moreover,  for $w \in H^{k+1}(\Omega)$,  for the projector $\Pi_h^*$, we have
[Lemma 1 of  \cite{Ste95}]
\begin{equation}\label{eqn20}
\|w- \Pi^*_h w\|_{1,h} \leq C h^k \|w\|_{k+1,\Omega}.\end{equation}
\noindent 
For the exact solution $\phi = \Delta u$, we get 
\begin{equation}\label{est11} 
\int_{\Omega} \phi \,q_h\,dx  - \langle \Delta_h q_h,u\rangle = 
\langle   \frac{\partial q_h}{\partial\bn} ,g_D \rangle_{\Gamma}  + 
\langle   g_N ,q_h \rangle_{\Gamma} ,\; q_h \in Q_h.
\end{equation}
Since $\Delta_h q_h \in M_h$, we have 
\[   \langle \Delta_h q_h,\Pi^*_h u\rangle =  \int_{\Omega} \Delta_h q_h\, u\,dx.\]
Thus we have
\begin{equation}\label{est12} 
\int_{\Omega}\Pi_h \phi \,q_h\,dx - \langle \Delta_h q_h,\Pi^*_h u\rangle = 
\langle   \frac{\partial q_h}{\partial\bn} ,g_D \rangle_{\Gamma} + 
\langle   g_N ,q_h \rangle_{\Gamma} ,\; q_h \in Q_h.
\end{equation}
Hence we have obtained that
$ (\Pi^*_hu, \Pi_h \phi) \in \CW_h$, and 
\[ \|(u-\Pi^*_h u,\phi-\Pi_h \phi)\|_{a} \leq 
C h^k \left(\|u\|_{k+1,\Omega}+|\phi|_{k,\Omega}\right).\]
The proof now follows from Lemma \ref{lemma2}.
\end{proof}
\begin{remark}
The existing error estimate approaches require 
an extra regularity of the solution $u$ \cite{Li99,Lam11a}. 
The energy error estimate in \cite{Lam11a,DP01} is sub-optimal even 
with the extra regularity, whereas the error estimate in \cite{Li99} is optimal 
but the approach works only  on rectangular meshes with a special structure. 
\end{remark}

\section{Algebraic formulation} \label{sec:alg}
To obtain an efficient numerical scheme in which 
all the auxiliary variables  (the vorticity $\phi_h$ and the Lagrange multiplier $p_h$) can be 
statically condensed out from the system, we construct 
a biorthogonal system for the sets of basis functions of $Q_h$ and $M_h$.  
Let $\{\varphi_1,\cdots,\varphi_n\}$ be a finite element basis for
the space $Q_h$. A finite element basis 
$\{\mu_1,\cdots,\mu_n\}$ for the space $M_h$ with 
$\supp \mu_i =\supp \varphi_i$, $1\leq i \leq n$,
is constructed in such a 
way  that the basis functions of $Q_h$ and 
$M_h$ satisfy a condition of biorthogonality relation
\begin{eqnarray} \label{eq:biorth}
  \int_{\Omega} \mu_i \ \varphi_j \, dx = c_j \delta_{ij},
\; c_j\neq 0,\; 1\le i,j \le n,
\end{eqnarray}
 where $n := \dim M_h = \dim Q_h$,  $\delta_{ij}$ is 
 the Kronecker symbol, and $c_j$ a scaling factor  proportional to the area $|\supp \phi_j|$.
 The basis functions of $M_h$ are constructed in a 
 reference element and they satisfy 
\eqref{ass0}, \eqref{ass1} and \eqref{eq:biorth} 
\cite{LSW05,Lam06,LW07}. 

Let $\vec{u}$, $\vec{\phi}$ and $\vec{p}$ be  the vector representations of 
the solution $u_h$, $\phi_h$ and $p_h$, respectively. 
Let $\tA\vec{u}$, $\tM\vec{\phi}$ and $\tD\vec{\phi}$ be algebraic representations of the bilinear 
forms $\int_{\Omega}  u_h \Delta_h q_h\,dx $, 
$\int_{\Omega}\phi_h \psi_h\,dx$ and  
$\int_{\Omega}\phi_h q_h\,dx$, respectively, where 
$u_h \in S_h$, $q_h \in Q_h$, $\phi_h,\psi_h \in M_h$.
We also denote the algebraic representation of the bilinear form
$\langle u_h,v_h\rangle_{\frac{1}{2},h} $ by 
$\tB_{\Gamma}\vec{u}$. Although the bilinear form $\langle u_h,v_h\rangle_{\frac{1}{2},h} $
is restricted to the boundary $\Gamma$ of the domain $\Omega$, $\tB_{\Gamma}$
is the extended form of the algebraic representation so that 
the number of columns of the 
matrix $\tB_{\Gamma}$ is equal to the number of components in $\vec{u}$, where entries of the matrix $\tB_{\Gamma}$ 
corresponding to interior nodes of the mesh are all set to zero. 
Then the  algebraic formulation of the saddle point problem \eqref{dbiharm}
is given by
\begin{equation} \label{biharmg}
\left[\begin{array}{cccc} 
\tB_{\Gamma}  & 0 & -\tA^T\\
 0 & \tM &\tD \\
-\tA & \tD & 0
 \end{array} \right]
\left[\begin{array}{ccc} \vec{u} \\ \vec{\phi} \\ \vec{p} 
\end{array}\right]=
\left[\begin{array}{ccc} \vec{f} \\0\\ \vec{g}  \end{array}\right],
\end{equation}
where $\vec{f}$ is the vector associated with 
the linear form $\ell_h(v_h)$, and $\vec{g}$ is the vector representation of 
$g_h(q_h)$.
Since the matrix $\tD$ is diagonal, we can do the static condensation of 
unknowns $\vec{\phi}$ and $\vec{p}$ and arrive at the following 
linear system  based on the unknown $\vec{u}$ associated only with the stream function: 
\begin{equation} \label{static}
\left(\tM_{\Gamma}+\tA^T \tD^{-1} \tM \tD^{-1}\tA\right) \vec{u} = 
(\vec{f}-(\tA^T \tD^{-1} \tM\tD^{-1}) \vec{g}).
\end{equation}
Since the inverse of the matrix $\tD$ is diagonal,
the system matrix in \eqref{static} is sparse.
It is important to have the system matrix to have sparse structure 
if an iterative solver is to be applied.
The vector corresponding to the vorticity $\vec{\phi}$ and the Lagrange multiplier $\vec{p}$
can be computed by simply inverting the
diagonal matrix using the second  and third blocks of \eqref{biharmg}. 

\section{Conclusion}
We have  proposed a finite element formulation for the biharmonic equation 
with clamped boundary conditions leading to an optimal convergence rate  improving the existing a priori error estimate in the energy norm.
The main idea is to impose the Dirichlet boundary condition weakly 
using the Nitsche technique. The new formulation also allows to use a biorthogonal system that gives an efficient finite element approach. In contrast to other Nitsche approaches, we do not require a penalty parameter in our formulation.

\section*{Acknowledgement}
Part of this work was completed during 
my visit to the Indian Institute of Technology, 
Mumbai in 2023. I gratefully acknowledge their hospitality. I especially 
thank my host Prof. Neela Nataraj so much for her wonderful hospitality and kindness 
 during my stay. I also thank Dr Devika Shylaja for carefully reading an earlier version of 
 this manuscript and providing many constructive comments.
\bibliographystyle{siam}
\bibliography{total}
\end{document}